\documentclass[11pt,a4paper,final]{amsart}

\usepackage[danish,english]{babel}
\usepackage{amsmath,amscd}
\usepackage{amssymb}
\usepackage{amsthm}
\usepackage{mathrsfs}
\usepackage{fixme}
\usepackage{enumerate}
\usepackage{graphicx}                    
\usepackage{pgf}
\usepackage{tikz}
\usetikzlibrary{arrows}
\usetikzlibrary{positioning}
\usetikzlibrary{shapes.geometric}
\usetikzlibrary{calc}
\usepackage[notref,notcite]{showkeys}
\usepackage{url}
\usepackage{eufrak}
\usepackage[vmargin={3.7cm,3.7cm},textwidth=.6\paperwidth]{geometry}

\theoremstyle{plain}
\newtheorem{theorem}{Theorem}[section]
\newtheorem{proposition}[theorem]{Proposition}
\newtheorem{lemma}[theorem]{Lemma}
\newtheorem{corollary}[theorem]{Corollary}

\theoremstyle{definition}
\newtheorem{definition}[theorem]{Definition}
\newtheorem{example}[theorem]{Example}
\newtheorem{remark}[theorem]{Remark}

\DeclareMathOperator{\id}{id}
\DeclareMathOperator{\Id}{Id}

\DeclareMathOperator{\BF}{BF}

\DeclareMathOperator{\diag}{diag}
\DeclareMathOperator{\sgn}{sgn}

\newcommand{\beg}{\mathfrak{b}}
\newcommand{\per}{\mathfrak{p}}

\newcommand{\N}{\mathbb{N}}
\newcommand{\Z}{\mathbb{Z}}

\newcommand{\R}{\mathbb{R}}

\renewcommand{\AA}{\mathcal{A}}
\newcommand{\BB}{\mathcal{B}}

\newcommand{\FF}{\mathcal{F}}
\newcommand{\LL}{\mathcal{L}}

\newcommand{\GG}{\mathcal{G}}

\newcommand{\X}{\mathsf{X}}

\newcommand{\FE}{\sim_{\textrm{FE}}}

\newcommand{\fracpart}[1]{ \langle #1 \rangle}

\newcommand{\intpart}[1]{ \lfloor #1 \rfloor}

\title{Flow equivalence of sofic beta-shifts} 
\author{Rune Johansen}
\date{\today}

\begin{document}

\begin{abstract}
The Fischer, Krieger, and fiber product covers of sofic beta-shifts are constructed and used to show that every strictly sofic beta-shift is $2$-sofic.
Flow invariants based on the covers are computed, and shown to only depend on an single integer easily determined from the $\beta$-expansion of 1.
It is shown that any beta-shift is flow equivalent to a beta-shift given by some $1< \beta < 2$, and concrete constructions lead to further reductions of the flow classification problem.
For each sofic beta-shift, there is an action of $\Z/2\Z$ on the edge shift given by the fiber product, and it is shown precisely when there exists a flow equivalence respecting these $\Z/2\Z$-actions.
This opens a connection to ongoing efforts to classify general irreducible 2-sofic shifts via flow equivalences of reducible SFTs equipped with $\Z/2\Z$-actions. 
\end{abstract}

\maketitle

\section{Introduction}
One of the simplest classes of shift spaces is the class of irreducible shifts of finite type, and Franks has given a very satisfactory classification of these up to flow equivalence in terms of a complete invariant that is both easy to compute and easy to compare \cite{franks}. This result has been extended to general shifts of finite type by Boyle and Huang  \cite{boyle,boyle_huang,huang}, but very little is known about the flow equivalence of the class of irreducible sofic shifts even though it constitutes a natural first generalization of the class of shifts of finite type.
There is an ongoing effort to classify irreducible 2-sofic shifts up to flow equivalence by considering the reducible edge shift given by the fiber product cover and the associated $\Z / 2\Z$ action \cite{boyle_carlsen_eilers_isotopy,boyle_carlsen_eilers_sofic}. The present work was motivated by a desire to apply these results to the class of sofic beta-shifts. Towards this end, canonical covers of sofic beta-shifts are constructed and used to compute a number of flow invariants. It is shown that strictly sofic beta-shifts are 2-sofic, 
and the flow classification of their fiber products is examined in detail. 
This serves to connect the present work to the machinery employed in \cite{boyle_carlsen_eilers_isotopy,boyle_carlsen_eilers_sofic}, but there is still a gap between the concrete constructions carried out in this paper and the general results of the classification program.
In another angle of attack, the flow equivalence problem for sofic beta-shifts is reduced to a simpler problem through a series of concrete constructions.

In Section \ref{sec_introduction}, notation is established and an introduction to the basic definitions and properties of beta-shifts is given. In Section \ref{sec_beta_covers}, the right Fischer covers of sofic  beta-shifts are determined, and the covering map is shown to be $2$ to $1$ if the shift is strictly sofic. This result is used to show that the right Krieger cover is identical to the right Fischer cover and to construct 
the right fiber product cover. Section \ref{sec_beta_classification} concerns the flow classification of beta-shifts. It is shown that for every $\beta > 1$, there exists $1 < \beta' < 2$ such that the two beta-shifts are flow equivalent and such that the new generating sequence has a form with certain properties. Additionally, the Bowen-Franks groups of the covers considered above are computed and shown to depend only on a sigle integer easily computed from the $\beta$-expansion of $1$.
There is an action of $\Z/2\Z$ on the reducible edge shift of the fiber product, and it is shown precisely when there exists a flow equivalence between the fiber product edge shifts that respects these $\Z/2\Z$-actions. 

\subsection*{Acknowledgements}
Supported by VILLUM FONDEN through the network for Experimental Mathematics in Number Theory, Operator Algebras, and Topology. Supported by the Danish National Research Foundation through the Centre for Symmetry and Deformation (DNRF92).

\section{Background and notation}
\label{sec_introduction}

\subsection{Shift spaces, flow equivalence, and labeled graphs}
Here, a short introduction to the definition and properties of shift
spaces is given to make the present paper self-contained; 
for a thorough treatment of shift spaces see \cite{lind_marcus}. 
Let $\AA$ be a finite set with the discrete topology. The
\emph{full shift} over $\AA$ consists of the space $\AA^\Z$ endowed
with the product topology and the \emph{shift map} $\sigma \colon
\AA^\Z \to \AA^\Z$ defined  by $\sigma(x)_i = x_{i+1}$ for all $i \in
\Z$. Let $\AA^*$ be the collection of finite words (also known as
blocks) over $\AA$. A
subset $X \subseteq \AA^\Z$ is called a \emph{shift space} if it
is invariant under the shift map and closed. 
For each $\FF \subseteq \AA^*$, define  $\X_\FF$ to be the set of
biinfinite sequences in $\AA^\Z$ which do not contain any of the
\emph{forbidden words} from $\FF$. 
A subset $X \subseteq \AA^\Z$ is a shift space if and only if there
exists $\FF \subseteq \AA^*$ such that $X = \X_\FF$
(cf. \cite[Proposition 1.3.4]{lind_marcus}). $X$ is said to be 
a \emph{shift of finite type} (SFT) if this is possible for a finite
set $\FF$.
The \emph{language} of a shift space $X$ is defined to be the set of all words which occur in at least one $x \in X$, and it is denoted $\BB(X)$.
For each $x \in X$, define the \emph{left-ray} of $x$ to be $x^- = \cdots
x_{-2} x_{-1}$ and define the \emph{right-ray} of $x$ to be $x^+ = x_0
x_1 x_2 \cdots$. For a shift space $X$, the sets of all left-rays and all right-rays are, respectively, denoted  $X^-$ and $X^+$. Define the \emph{predecessor set} of a right ray $x^+ \in X^+$ to be $P^\infty(x^+) = \{ x^- \in X^- \mid x^-x^+\in X \}$.  

A bijective, continuous, and shift intertwining map between two shift
spaces is called a \emph{conjugacy}, and when such a map exists, 
the two shift spaces are said to be \emph{conjugate}. 
A function $\pi \colon X_1 \to X_2$ between shift spaces $X_1$ and
$X_2$ is said to be a \emph{factor map} if it is continuous,
surjective, and shift intertwining. A shift space is
called \emph{sofic} \cite{weiss} if it is the image of an SFT under a factor
map.

Let $(X ,\sigma)$ be a shift space, equip $X \times \R$ with the product topology, and define an equivalence relation on $X \times \R$ by $(x,t) \sim (\sigma(x), t-1)$. The \emph{suspension flow} $SX$ of $X$ is the quotient space $X \times \R / {\sim}$. The equivalence class of $(x,t)$ in $SX$ will be denoted $[x,t]$. For each $[x,t] \in SX$ and $r \in \R$, define $[x,t] + r = [x,t+r]$. For each $z \in SX$, the set $\{ z + r \in SX \mid r \in \R \}$ is called a \emph{flow line}.
If $Y$ is a shift space and $\Phi \colon SX \to SY$ is a homeomorphism, then for each $z \in SX$, there exists a map $\varphi_z \colon \R \to \R$ such that $\Phi(z + r) = \Phi(z) + \varphi_{z}(r)$ for all $r \in \R$, i.e.\ a homeomorphism maps flow lines to flow lines.
A homeomorphism $\Phi \colon SX \to SY$ is said to be a \emph{flow equivalence} if
there for each $z \in SX$ exists a monotonically increasing map $\varphi_z \colon \R \to \R$ such that $\Phi(z + r) = \Phi(z) + \varphi_{z}(r)$. In this case, $X$ and $Y$ are said to be \emph{flow equivalent} and this is denoted $X \FE Y$. Flow equivalence is generated by conjugacy
and \emph{symbol expansion} \cite{parry_sullivan}. In the following, a number of lemmas from \cite[Section 2.3]{johansen_thesis} will be used in the construction of concrete flow equivalences between beta-shifts.

For countable sets $E^0$ and $E^1$, and maps $r,s \colon E^1 \to E^0$
the quadruple $E = (E^0,E^1,r,s)$ is called a \emph{directed graph}. The
elements of $E^0$ and $E^1$ are, respectively, the vertices and the
edges of the graph.
For each edge $e \in E^1$, $s(e)$ is the vertex where $e$ starts,
and $r(e)$ is the vertex where $e$ ends. A \emph{path} $\lambda = e_1 \cdots
e_n$ is a sequence of edges such that $r(e_i) = s(e_{i+1})$ for all $i
\in \{1, \ldots n-1 \}$.
For each $n \in \N_0$, the set of paths of length $n$ is denoted
$E^n$, and the set of all finite paths is denoted $E^*$.
Extend
the maps $r$ and $s$ to $E^*$ in the natural way.
$E$ is said to be \emph{essential} if every vertex emits
and receives an edge.
For a finite essential directed graph $E$, the 
\emph{edge shift} $(\X_E, \sigma_E)$ is defined by
\begin{displaymath}
  \X_E = \left\{ x \in (E^1)^\Z \mid r(x_i) = s(x_{i+1}) \textrm{ for all }
  i \in \Z \right\}.
\end{displaymath}
For a finite graph $E$ with $n \times n$ integer adjacency matrix $A$, the \emph{Bowen--Franks} group of $\X_E$ is the co-kernel $\BF(\X_E) = \Z^n / \Z^n (\Id - A)$ \cite[Definition 7.4.15]{lind_marcus}.
If $\X_A, \X_B$ are irreducible edge shifts not flow equivalent to the trivial shift then $\X_A \FE \X_B$ if and only if $\BF(\X_A) = \BF(\X_B)$ and $\sgn \det(\Id - A) = \sgn \det(\Id - B)$ \cite{franks}. The pair consisting of the Bowen--Franks group and the sign of the determinant will be denoted $\BF_+$.

A \emph{labeled graph} $(E, \LL)$ over an alphabet $\AA$ consists
of a directed graph $E$ and a surjective labeling map $\LL \colon E^1
\to \AA$. Extend the labeling map to $\LL \colon E^* \to \AA^*$ by
defining $\LL(e_1 
\cdots e_n) = \LL(e_1) \cdots \LL(e_n) \in \AA^*$.
For a finite essential labeled graph $(E, \LL)$, define the shift space $(\X_{(E, \LL)}, \sigma)$ by 
\begin{displaymath}
  \X_{(E, \LL)} = \left\{ \left( \LL(x_i) \right)_i \in \AA^\Z \mid 
                          x \in \X_E  \right\}.
\end{displaymath}
The
labeled graph $(E, \LL)$ is said to be a \emph{presentation} of the
shift space $\X_{(E, \LL)}$, and a \emph{representative} of a word $w \in
\BB(\X_{(E, \LL)}) $ is a path $\lambda \in E^*$ such that
$\LL(\lambda) = w$.

Every SFT is sofic, and a sofic shift which is not an SFT is called \emph{strictly sofic}.
A shift space is sofic if and only 
if it can be presented by a finite labeled graph \cite{fischer}.
An irreducible sofic shift $X$, has a unique (up to graph automophism) minimal follower-separated presentation called the right \emph{Fischer cover} \cite{fischer}.
A strictly sofic shift is said to be \emph{2-sofic} if the factor map induced by the labeling of the right Fischer cover is $2$ to $1$.  
The right \emph{Krieger cover} gives a canonical presentation of an arbitrary sofic shift \cite{krieger_sofic_I}.

\subsection{Beta-shifts}
\label{sec_beta_introduction}
Here, a short introduction to the basic definitions and properties of beta-shifts is given. For a more detailed treatment of beta-shifts, see \cite{blanchard_beta}.
Let $\beta \in \R$ with $\beta > 1$. For each $t \in [0;1]$ define sequences $(r_n(t))_{n \in \N}$ and $(x_n(t))_{n \in \N}$ by 
\begin{align*}
     r_1(t)  &=  \fracpart{ \beta t }   ,   &r_n(t) = \fracpart{ \beta r_{n-1}(t) }, \\
     x_1(t) &=  \intpart{ \beta t }   ,     &x_n(t) = \intpart{ \beta r_{n-1}(t) },
\end{align*}
where $\intpart y$ and $\fracpart y$ are respectively the integer part and the fractional part of $y \in \R$. The sequence $x_1(t)x_2(t) \cdots$ is said to be the $\beta$-\emph{expansion} of $t$, and R{\'e}nyi \cite{renyi} has proved that  
\begin{displaymath}
  t = \sum_{n=1}^\infty \frac{ x_n(t) }{ \beta^n }.
\end{displaymath} 
The $\beta$-expansion of $1$ is denoted $e(\beta)$. As an example, consider $\beta = (1+\sqrt 5)/2$ where the $\beta$-expansion of $1$ is $e(\beta) = 11000\cdots$. 
The $\beta$-expansion of $t$ is said to be \emph{finite} if there exists $N \in \N$ such that $x_n(t) = 0$ for all $n \geq N$. It is said to be \emph{eventually periodic} if there exist $N,p \in \N$ such that $x_{n+p}(t) = x_n(t)$ for all $n \geq N$.
Define the \emph{generating sequence of} $\beta$ to be
\begin{displaymath}
g(\beta) = \left \{ \begin{array}{l c l}
        e(\beta)                     & , & e(\beta) \textrm{ is infinite } \\
        (a_1 a_2 \cdots a_{k-1}(a_k-1))^\infty   & , & e(\beta) = a_1  \cdots a_k00\cdots \textrm{ and } a_k \neq 0
      \end{array} \right. .
\end{displaymath}

\index{beta-shift}
\index{golden mean shift!as beta-shift}
\index{beta-shift!order on}
Define $\BB_\beta = \{ x_n(t) \cdots x_m(t) \mid 1 \leq n \leq m, t \in [0;1] \}$. It is easy to check that $\BB_\beta$ is the language of a shift space $\X_\beta$. Such a shift space is called a \emph{beta-shift}. The alphabet $\AA_\beta$ of $\X_\beta$ is either $\{ 0, \ldots, \beta -1\}$ or $ \{0 , \ldots , \intpart \beta \}$ depending on whether $\beta$ is an integer. If $\beta \in \N$, then $\X_\beta$ is the full $\{0, \ldots, \beta-1\}$-shift. For $\beta = (1+\sqrt 5)/2$ as considered above, $\X_\beta$ is conjugate to the golden mean shift.

The words in $\BB(\X_\beta)$ and right-rays in $\X_\beta^+$ are ordered by lexicographical order $\leq$, and the following two theorems use this order to give fundamental descriptions of beta-shifts.

\begin{theorem}[{R{\'e}nyi \cite{renyi}}] 
\label{thm_renyi}
Let $\beta >1$, let $g(\beta)$ be the generating sequence, let $\AA_\beta$ be the alphabet of $\X_\beta$, and let $x^+ = x_1 x_2 \ldots \in \AA_\beta^\N$. Then $x^+ \in \X_\beta^+$ if and only if 
$x_k x_{k+1} \cdots \leq g(\beta)$ for all $k \in \N$.
\end{theorem}

\begin{theorem}[{Parry \cite{parry}}]
\label{thm_beta_existence}
A sequence $a_1 a_2 \cdots$ is the $\beta$-expansion of $1$ for some $\beta > 1$ if and only if $a_k a_{k+1} \cdots < a_1 a_2 \cdots$ for all $k \in \N$. Such a $\beta$ is uniquely given by the expansion of $1$.
\end{theorem}

For detailed treatments of the connections that beta-shifts provide between symbolic dynamics and number theory, see e.g.\ \cite{schmidt, bertrand-mathis,parry,  denker_grillenberger_sigmund,lind_beta}.

\section{Covers of sofic beta-shifts}
\label{sec_beta_covers}

Let $\beta >1$, let $g(\beta) = g_1 g_2 \cdots$ be the generating sequence of $\X_\beta$, and define an infinite labeled graph $\GG_\beta = (G_\beta, \LL_\beta)$ with  vertices $G_\beta^0 = \{ v_i \mid i \in \N \}$ and edges $G^1 = \{ e_i^{k} \mid i \in \N, 0 \leq k \leq g_i \}$ such that  
\begin{displaymath}
s(e_i^k) = v_i, \qquad
 r(e_i^k) = \left \{ \begin{array}{l c l}
      v_{i+1} & ,  & k = g_i \\
      v_1       & , & k < g_i
\end{array} \right.,
\qquad \textrm{and} \qquad
\LL_\beta(e_i^k) = k
\end{displaymath}
for all $i \in \N$ and $0 \leq k \leq g_i$.
It is easy to check that this is a right-resolving presentation of $\X_\beta$.
The labeled graph $(G_\beta, \LL_\beta)$ is called the \emph{standard loop graph presentation} of $\X_\beta$, and it is sketched in  Figure \ref{fig_beta_loop}. Note that the structure of the standard loop graph shows that every beta-shift is irreducible. 

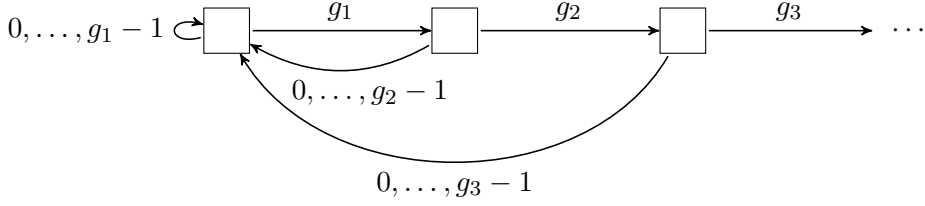
\begin{figure}
\begin{center}
\begin{tikzpicture}
  [bend angle=10,
   clearRound/.style = {circle, inner sep = 0pt, minimum size = 17mm},
   clear/.style = {rectangle, minimum width = 10 mm, minimum height = 6 mm, inner sep = 0pt},  
   greyRound/.style = {circle, draw, minimum size = 1 mm, inner sep =
      0pt, fill=black!10},
   grey/.style = {rectangle, draw, minimum size = 6 mm, inner sep =
      1pt, fill=black!10},
   white/.style = {rectangle, draw, minimum size = 6 mm, inner sep =
      1pt},
   to/.style = {->, shorten <= 1 pt, >=stealth', semithick}]
  
  \node[white] (P1) at (0,0) {};
  \node[white] (P2) at (3,0) {};   
  \node[white] (P3) at (6,0) {};
  \node[clear] (P4) at (9,0) {$\cdots$};

  \draw[to] (P1) to node[auto] {$g_1$} (P2); 
  \draw[to] (P2) to node[auto] {$g_2$} (P3);
  \draw[to] (P3) to node[auto] {$g_3$} (P4);

  \draw[to, loop left] (P1) to node[auto] {$0, \ldots , g_1-1$} (P1); 
  \draw[to, bend left = 30] (P2) to node[auto] {\qquad $0, \ldots , g_2-1$} (P1);     
  \draw[to, bend left = 60] (P3) to node[auto] {$0, \ldots , g_3-1$} (P1); 
  
\end{tikzpicture}
\end{center}
\caption[Standard loop graphs of beta-shifts.]{The standard loop graph of a beta-shift with generating sequence $g(\beta) = (g_n)_{n\in \N}$. For the sake of clarity, an edge labeled $0,\ldots,g_{i}-1$ terminating at the leftmost vertex in the figure represents $g_i$ individual edges, each labeled with one of the symbols $0,\ldots,g_{i}-1$.} 
\label{fig_beta_loop}
\end{figure}

\subsection{Fischer cover}

Parry proved that $\X_\beta$ is an SFT if and only if the generating sequence is periodic \cite{parry}, and Betrand-Mathis  proved that $\X_\beta$ is sofic if and only if the generating sequence is eventually periodic \cite{bertrand-mathis_preprint}. The latter result is apparently only available in a preprint, so an argument for this fact is given as part of the proof of the following Proposition.

\begin{proposition}
\label{prop_beta_sofic}
Given $\beta > 1$, the beta-shift $\X_\beta$ is sofic if and only if the generating sequence $g(\beta)$ is eventually periodic. For minimal $n,p \in \N$ with $g(\beta) = g_1 \cdots g_n (g_{n+1} \cdots g_{n+p})^\infty$, the right Fischer cover of $\X_\beta$ is the labeled graph $(F_\beta, \LL_\beta)$ shown in Figure \ref{fig_beta_rfc} which has $F_\beta^0 = \{ v_1, \ldots v_{n+p} \}$, $F_\beta^1 = \{ e_i^k \mid 1 \leq i \leq n+p, 0 \leq k \leq g_i \}$, and 
\begin{displaymath}
s(e_i^k) = v_i, \textrm{  } 
 r(e_i^k) = \left \{ \begin{array}{l c l}
      v_1       & ,  & k < g_i \\ 
      v_{i+1}  & ,  & k = g_i , i < n+p \\
      v_{n+1} & ,  & k = g_{n+p} , i = n+p      
\end{array} \right.,
\textrm{ and } 
\LL_\beta(e_i^k) = k. 
\end{displaymath}
\end{proposition}

\noindent
Note that the right Fischer cover $(F_\beta, \LL_\beta )$ of a sofic beta-shift with generating sequence $g(\beta) = g_1 \cdots g_n (g_{n+1} \cdots g_{n+p})^\infty$ is the labeled graph obtained from the subgraph of the standard loop graph $\GG_\beta$ induced by the first $n+p$ vertices $v_1, \ldots , v_{n+p}$ of $G_\beta^0$ by adding an additional edge labeled $g_{n+p}$ from $v_{n+p}$ to $v_{n+1}$. 

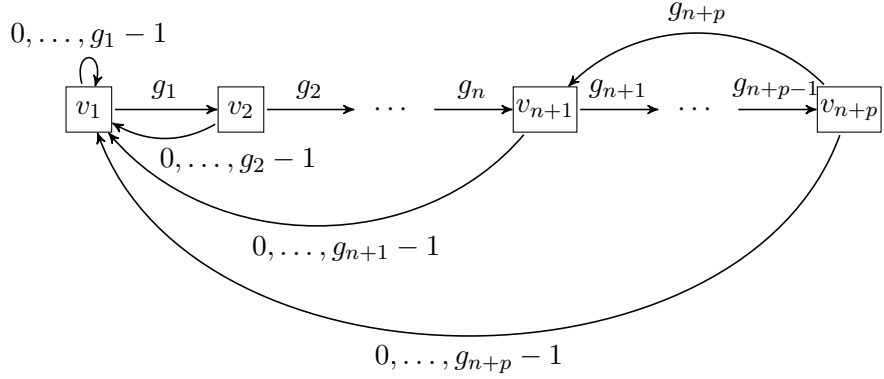
\begin{figure}
\begin{center}
\begin{tikzpicture}
  [bend angle=10,
   clearRound/.style = {circle, inner sep = 0pt, minimum size = 17mm},
   clear/.style = {rectangle, minimum width = 10 mm, minimum height = 6 mm, inner sep = 0pt},  
   greyRound/.style = {circle, draw, minimum size = 1 mm, inner sep =
      0pt, fill=black!10},
   grey/.style = {rectangle, draw, minimum size = 6 mm, inner sep =
      1pt, fill=black!10},
   white/.style = {rectangle, draw, minimum size = 6 mm, inner sep =
      1pt},
   to/.style = {->, shorten <= 1 pt, >=stealth', semithick}]
  
  \node[white] (v1)     at (0,0) {$v_1$};
  \node[white] (v2)     at (2,0) {$v_2$};
  \node[clear] (dots)   at (4,0) {$\cdots$};
  \node[white] (vn+1)     at (6,0) {$v_{n+1}$};
  \node[clear] (dots2) at (8,0) {$\cdots$}; 
  \node[white] (vn+p)     at (10,0) {$v_{n+p}$};

  \draw[to] (v1) to node[auto] {$g_1$} (v2); 
  \draw[to] (v2) to node[auto] {$g_2$} (dots); 
  \draw[to] (dots) to node[auto] {$g_n$} (vn+1);
  \draw[to] (vn+1) to node[auto] {$g_{n+1}$} (dots2);
  \draw[to] (dots2) to node[auto] {$g_{n+p-1}$} (vn+p);  

  \draw[to, loop above] (v1) to node[auto] {$0, \ldots , g_1-1$} (v1); 
  \draw[to, bend left = 30] (v2) to node[auto] {\quad \qquad \qquad $0, \ldots , g_{2}-1$} (v1);    
  \draw[to, bend left = 50] (vn+1) to node[auto] {\qquad $0, \ldots , g_{n+1}-1$} (v1);     
  \draw[to, bend left = 70] (vn+p) to node[auto] {$0, \ldots , g_{n+p}-1$} (v1);     
  \draw[to, bend right = 45] (vn+p) to node[auto,swap] {$g_{n+p}$} (vn+1);     
  
\end{tikzpicture}
\end{center}
\caption[Right Fischer covers of sofic beta-shifts.]{Right Fischer cover of a sofic beta-shift with generating sequence $g(\beta) = g_1 \cdots g_n (g_{n+1} \cdots g_{n+p})^\infty$ for minimal $n,p$.} 
\label{fig_beta_rfc}
\end{figure}

\begin{proof}[Proof of Proposition \ref{prop_beta_sofic}]
Assume that $g(\beta)$ is not eventually periodic and let $v,w$ be finite prefixes of $g(\beta)$ with $v \neq w$. Choose $x_v^+,x_w^+ \in \X_\beta^+$ such that $vx_v^+ = g(\beta) = wx_w^+$.
Since $g(\beta)$ is not eventually periodic, $x_v^+ \neq x_w^+$. Assume without loss of generality that $x_v^+ > x_w^+$. Then $w x_v^+ > wx_w^+ = g(\beta)$, so by Theorem \ref{thm_renyi}, $w x_v^+ \notin \X_\beta^+$. Hence, $P_\infty(x_v^+) \neq P_\infty(x_w^+)$. Since $g(\beta)$ is not eventually periodic, this proves that $\X_\beta$ has infinitely many different predecessor sets, so it is not a sofic shift 
\cite[\S 2]{krieger_sofic_I}.
The standard loop graph is a right-resolving presentation of $\X_\beta$, so it follows from the observation preceding this proof that $(F_\beta, \LL_\beta )$ is so as well. It is easy to check that $(F_\beta, \LL_\beta)$ is also follower-separated, so it is the right Fischer cover of $\X_\beta$.
\end{proof}

\begin{lemma}[Johnson {\cite[Proposition 2.5.1]{johnson_thesis}}]
\label{lem_beta_w_unique_r}
Let $\beta > 1$, let $\X_\beta$ be sofic, and let $(F_\beta, \LL_\beta)$ be the right Fischer cover of $\X_\beta$. If there is a path $\lambda \in F_\beta^*$ with $s(\lambda) = v_1$ such that $\LL_\beta(\lambda) = w$ is not a factor of $g(\beta)$, then every path in $F_\beta^*$ labeled $w$ terminates at $r(\lambda)$.
\end{lemma}

\noindent
Some cases were left unchecked in the proof in \cite{johnson_thesis}, but a complete proof is given in \cite[Lemma 4.6]{johansen_thesis}.

\begin{proposition}
\label{prop_beta_2_to_1}
Let $\beta > 1$ with eventually periodic $g(\beta)$, let $(F_\beta, \LL_\beta)$ be the right Fischer cover  of $\X_\beta$, and let $\pi \colon \X_{F_\beta} \to \X_\beta$ be the covering map. If $g(\beta)$ is periodic, then $\pi$ is $1$ to $1$, and if not, then it is $2$ to $1$.
\end{proposition}

\begin{proof}
Let $g(\beta) =  g_1 \cdots g_n (g_{n+1} \cdots g_{n+p})^\infty$ for minimal $n,p \in \N$, and let $\per = g_{n+1} \cdots g_{n+p}$. The first goal is to prove that if $x \in \X_\beta$ has more than one presentation in $(F_\beta, \LL_\beta)$, then there exits $k \in \Z$ such that $\cdots x_{k-1} x_k = \per^\infty$. 

Let $x \in \X_\beta$, and let $\lambda$ be a biinfinite path in $F_\beta$ with $\LL_\beta(\lambda) = x$. 
If there is a lower bound $l$ on the set 
\begin{displaymath}
A = \{ j \in \Z \mid \exists i < j : \LL_\beta(\lambda_{[i,j]}) \textrm{ is not a factor of } g(\beta) \},
\end{displaymath}
then there exists $k \leq l$ such that $\cdots x_{k-1} x_k = \per^\infty$ since $x$ is biinfinite while $g(\beta)$ is not.
By Proposition \ref{prop_beta_sofic},
the only circuit in $(F_\beta, \LL_\beta)$ that does not pass through $v_1$ is labeled $\per$, so if there is a lower bound $l$ on the set $B = \{ i \in \Z \mid r(\lambda_i) = v_1 \}$, then there exists $k < l$ such that $\cdots x_{k-1} x_k = \per^\infty$.
Assume that both $A$ and $B$ are unbounded below, let $\mu \in \X_{F_\beta}$ with $\LL(\mu) = x$, and let $k \in \Z$. Then there exist $i < j < k$ with $s(\lambda_i) = v_1$ such that $w = \LL_\beta(\lambda_{[i,j]})$ is not a factor of $g(\beta)$. By Lemma \ref{lem_beta_w_unique_r}, $r(\mu_j) = r(\lambda_j)$, and $(F_\beta, \LL_\beta)$ is right-resolving, so this implies that $\mu_k = \lambda_k$. Since $k$ was arbitrary, $\lambda = \mu$, and this proves the claim.
 
If $g(\beta)$ is periodic, then every circuit in $F_\beta$ passes through $v_1$, and only one of these is labeled $\per$ since $p$ is the minimal period. Hence, there is precisely one vertex in $F^0$ where a presentation of $\per^\infty$ can end, so $\pi_\beta$ is $1$ to $1$ and $\X_\beta$ is an SFT.

Assume that $g(\beta)$ is eventually periodic without being periodic. Then $\per < g(\beta)_{[1,p]}$, so there is a circuit $\mu$ in $F_\beta$ passing through $v_1$ with $\LL_\beta(\mu) = \per$. Choose $0 \leq i \leq p$ such that $\per_i \cdots \per_p \per_1 \cdots \per_{p-1}$ is maximal among the cyclic permutations of the letters of $\per$. The number $i$ is unique because $p$ is the minimal period. Now $s(\mu_i) = v_i$, so $\mu$ is unique because $(F_\beta, \LL_\beta)$ is right-resolving.
The only circuit that does not pass through $v_1$ is also labeled $\per$, so $\pi_\beta$ is $2$ to $1$. 
\end{proof}

\begin{figure}
\begin{center}
\begin{tikzpicture}
  [bend angle=10,
   clearRound/.style = {circle, inner sep = 0pt, minimum size = 17mm},
   clear/.style = {rectangle, minimum width = 17 mm, minimum height = 6 mm, inner sep = 0pt},  
   greyRound/.style = {circle, draw, minimum size = 1 mm, inner sep =
      0pt, fill=black!10},
   grey/.style = {rectangle, draw, minimum size = 6 mm, inner sep =
      1pt, fill=black!10},
   white/.style = {rectangle, draw, minimum size = 6 mm, inner sep =
      1pt},
   to/.style = {->, shorten <= 1 pt, >=stealth', semithick}]
  
  \node[white] (P1) at (0,0) {};
  \node[white] (P2) at (2,0) {};   
  \node[white] (P3) at (4,0) {};
  \node[white] (P4) at (6,0) {};

  \draw[to] (P1) to node[auto] {$1$} (P2); 
  \draw[to] (P2) to node[auto] {$1$} (P3); 
  \draw[to] (P3) to node[auto] {$1$} (P4);

  \draw[to, loop left] (P1) to node[auto] {$0$} (P1);
  \draw[to, bend left = 30] (P2) to node[auto] {$0$} (P1);
  \draw[to, bend left = 60] (P3) to node[auto] {$0$} (P1);
  \draw[to, bend right = 45] (P4) to node[auto,swap] {$0$} (P3);  
   
\end{tikzpicture}
\end{center}
\caption[Right Fischer cover of a sofic beta-shift.]{Right Fischer cover of the sofic beta-shift from Example \ref{ex_beta_2_sofic}.} 
\label{fig_beta_2_sofic}
\end{figure}
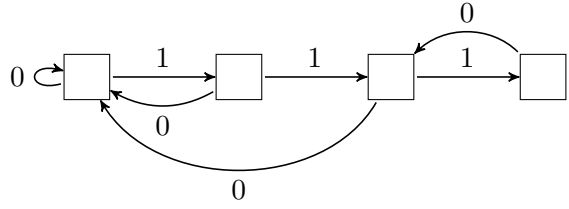

\begin{example}
\label{ex_beta_2_sofic}
Use Theorem \ref{thm_beta_existence} to choose $\beta > 1$ such that $g(\beta) = 11(10)^\infty$. The right Fischer cover of $\X_\beta$ is shown in Figure \ref{fig_beta_2_sofic}. Note that there are two presentations of e.g.\ the biinfinite sequence $(10)^\infty$.
\end{example}

The following result was proved by Parry \cite{parry}, but it is repeated here since it follows immediately from Propositions \ref{prop_beta_sofic} and \ref{prop_beta_2_to_1}.

\begin{corollary}
\label{cor_beta_sft}
For $\beta > 1$, the beta-shift $\X_\beta$ is an SFT if and only if the generating sequence $g(\beta)$ is periodic.
\end{corollary}

\subsection{Krieger cover}

\begin{proposition}
\label{prop_beta_krieger}
If $\beta > 1$ and the beta-shift $\X_\beta$ is sofic, then the right Krieger cover of $\X_\beta$ is identical to the right Fischer cover $(F_\beta, \LL_\beta)$ of $\X_\beta$.
\end{proposition}

\begin{proof}
Let $n,p$ be minimal such that $g(\beta) = g_1 \cdots g_n(g_{n+1} \cdots g_{n+p})^\infty$ and let $F_\beta^0 = \{ v_1 , \ldots, v_{n+p} \}$ as in Proposition \ref{prop_beta_sofic}.
Let $x^- \in \X_\beta^-$, and let $r(x^-) \subseteq F_\beta^0$ be the set of vertices where a presentation of $x^-$ can end. The goal is to prove that $x^-$ is $1$-synchronizing (see \cite{johansen_structure}).

Assume first that $x^- = (g_{n+1} \cdots g_{n+p})^\infty g_{n+1} \cdots g_{n+k}$ for some $k \leq p$. By the proof of Proposition \ref{prop_beta_2_to_1}, there is a unique $i$ such that there is a circuit labeled $\per' = g_{n+k+1} \cdots g_{n+p} g_{n+1} \cdots g_{n+k}$ passing through $v_1$ and terminating at $v_i$. Similarly, there is a unique $j \neq i$  such that there is a circuit labeled $\per'$ which terminates at $v_j$ without passing through $v_1$. Now, $r(x^-) = \{ v_i, v_j \}$ and $F_\infty(v_j) \subseteq F_\infty(v_i)$, so $x^-$ is $1$-synchronizing.

If $x^-$ is not of the form considered above, then the proof of Proposition \ref{prop_beta_2_to_1} shows that $\lvert r(x^-) \rvert = 1$, so $x^-$ is $1$-synchronizing. Hence, the right Krieger cover is equal to the right Fischer cover by \cite[Section 4]{johansen_structure}.
\end{proof}

\noindent
This result also follows from \cite{katayama_matsumoto_watatani} where it is shown that the Matsumoto algebra associated to $\X_\beta$ is simple.

\subsection{Fiber product}
By Proposition \ref{prop_beta_2_to_1}, the covering map of the Fischer cover of a sofic beta-shift is $2$ to $1$. There is an ongoing effort to classify irreducible $2$-sofic shifts up to flow equivalence via flow equivalence of certain derived reducible shift spaces equipped with an action of $\Z / 2\Z$ \cite{boyle_carlsen_eilers_isotopy,boyle_carlsen_eilers_sofic}. These tools are not yet general enough to be applied to beta-shifts, but the present work was motivated by a desire to pave the way for such an investigation.
This section contains an introduction to fiber products and a construction of the right fiber product covers of sofic beta-shifts.

\begin{definition}
\index{fiber product graph}
Let $X$ be an irreducible sofic shift and let $(F, \LL_F)$ be the right Fischer cover of $X$. The \emph{(right) fiber product graph} of $X$ is defined to be the labeled graph with vertex set $\{ (u,v) \mid u,v \in F^0 \}$ where there is an edge labeled $a$ from $(u_1, v_1)$ to $(u_2, v_2)$ if and only if there are edges labeled $a$ from $u_1$ to $u_2$ and from $v_1$ to $v_2$ in $F$.
\end{definition}

Let $X$ be a sofic shift over $\AA$, let $(F, \LL_F)$ be the right Fischer cover of $X$, let $n = \lvert F^0 \rvert$, and let $A$ be the corresponding symbolic adjacency matrix.
Then the symbolic adjacency matrix of the fiber product graph is $(B(i,j))_{1 \leq i,j \leq n}$ where $B(i,j)$ is the $n \times n$ matrix over formal sums over $\AA$ obtained from $A$ by omitting all symbols not appearing in the entry $A_{ij}$.

The fiber product graph of $X$ is a presentation of $X$, and it contains the right Fischer cover as the subgraph induced by the diagonal vertices $\{ (v,v) \mid v \in F^0 \}$. The fiber product graph is generally not essential, so it is often useful to pass to the maximal essential subgraph. This subgraph will be called the \emph{fiber product cover} in the following.

Let $X$ be a sofic shift with right Fischer cover $(F, \LL_F)$ such that the covering map $\pi \colon \X_F \to X$ is $2$ to $1$, and let $(P, \LL_P)$ be the fiber product cover of $X$. 
Let $\lambda \in \X_P$ be a biinfinite sequence, let $i \in \Z$, let $s(\lambda_i) = (u_i,v_i) \in P^0$, and let $ r(\lambda_i) = (u_{i+1},v_{i+1}) \in P^0$. Then $(v_i,u_i), (v_{i+1},u_{i+1}) \in P^0$ as well, and there is a unique edge labeled $\LL_P(\lambda_i)$ from $(v_i,u_i)$ to $(v_{i+1},u_{i+1})$. This defines a continuous and shift commuting map $\varphi \colon \X_P \to \X_P$ with $\varphi^2 = \id$. In this way, the labeling induces a continuous and shift commuting $\Z / 2\Z$ action on the edge shift $\X_P$. This is said to be the $\Z / 2\Z$ action on $\X_P$ \emph{induced by the labels}. This also induces a corresponding continuous $\Z / 2\Z$ action on the suspension $S\X_P$. See \cite{boyle_sullivan} for an investigation of equivariant flow equivalence of shift spaces equipped with group actions.

\begin{example}
\label{ex_beta_fiber}
Use Theorem \ref{thm_beta_existence} to find $1 < \beta < 2$ such that $g(\beta) = 1(10)^\infty$. The symbolic adjacency matrix of the right Fischer cover of the corresponding beta-shift $\X_\beta$ is
\begin{displaymath}
\left( \begin{array}{ c c c }
0 & 1 &      \\
0 &    & 1   \\
   & 0 &      \\
\end{array} \right), 
\end{displaymath}
where a blank entry signifies that there is no edge between the two vertices.
It is straightforward to check that the labeled graph in Figure \ref{fig_beta_fiber_ex} is the fiber product cover of this beta-shift.
\end{example}

\begin{figure}
\begin{center}
\begin{tikzpicture}
  [bend angle=10,
   clearRound/.style = {circle, inner sep = 0pt, minimum size = 17mm},
   clear/.style = {rectangle, minimum width = 10 mm, minimum height = 6 mm, inner sep = 0pt},  
   greyRound/.style = {circle, draw, minimum size = 1 mm, inner sep =
      0pt, fill=black!10},
   grey/.style = {rectangle, draw, minimum size = 6 mm, inner sep =
      1pt, fill=black!10},
   white/.style = {rectangle, draw, minimum size = 6 mm, inner sep =
      1pt},
   to/.style = {->, shorten <= 1 pt, >=stealth', semithick}]
  
  \node[white] (P1) at (0,0) {$(1,1)$};
  \node[white] (P2) at (2,0) {$(2,2)$};   
  \node[white] (P3) at (4,0) {$(3,3)$};

  \node[white] (P12) at (0,1.5) {$(1,2)$};
  \node[white] (P23) at (2,1.5) {$(2,3)$};   

  \node[white] (P21) at (0,-1.5) {$(2,1)$};
  \node[white] (P32) at (2,-1.5) {$(3,2)$};

  \draw[to] (P1) to node[auto] {$1$} (P2); 
  \draw[to, loop left] (P1) to node[auto] {$0$} (P1);   
  \draw[to] (P2) to node[auto] {$1$} (P3);
  \draw[to, bend left = 30] (P2) to node[auto] {$0$} (P1);  
  \draw[to, bend right = 45] (P3) to node[auto,swap] {$0$} (P2);    

  \draw[to,bend left] (P12) to node[auto] {$1$} (P23); 
  \draw[to,bend left] (P23) to node[auto] {$0$} (P12);
  \draw[to] (P12) to node[auto,swap] {$0$} (P1);   
   
  \draw[to,bend left] (P21) to node[auto] {$1$} (P32); 
  \draw[to,bend left] (P32) to node[auto] {$0$} (P21);
  \draw[to] (P21) to node[auto] {$0$} (P1);   

\end{tikzpicture}
\end{center}
\caption[Fiber product cover of a sofic beta-shift.]{Fiber product cover of the sofic shift from Example \ref{ex_beta_fiber}.} 
\label{fig_beta_fiber_ex}
\end{figure}
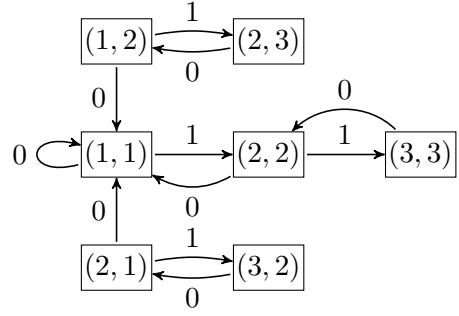

\begin{figure}
\begin{center}
\begin{tikzpicture}
 [bend angle=10,
   clearRound/.style = {circle, inner sep = 0pt, minimum size = 17mm},
   clear/.style = {rectangle, minimum width = 10 mm, minimum height = 6 mm, inner sep = 0pt},  
   greyRound/.style = {circle, draw, minimum size = 1 mm, inner sep =
      0pt, fill=black!10},
   grey/.style = {rectangle, draw, minimum size = 6 mm, inner sep =
      1pt, fill=black!10},
   white/.style = {rectangle, draw, minimum size = 6 mm, inner sep =
      1pt},
   to/.style = {->, shorten <= 1 pt, >=stealth', semithick}]
  
  \node[white] (v1)     at (0,0) {$v_1$};
  \node[white] (v2)     at (1.5,0) {$v_2$};
  \node[clear] (dots)   at (3,0) {$\cdots$};
  \node[white] (vn+1) at (5,0) {$v_{n+1}$};
  \node[clear] (dots2) at (7,0) {$\cdots$}; 
  \node[white] (vn+p) at (9,0) {$v_{n+p}$};

  \node[white] (vn+1')     at (0,3) {$v_{n+1}'$};
  \node[clear] (dots2')     at (2,3) {$\cdots$}; 
  \node[white] (vn+p')     at (4,3) {$v_{n+p}'$};

  \node[white] (vn+1'')     at (0,-5) {$v_{n+1}''$};
  \node[clear] (dots2'')     at (2,-5) {$\cdots$}; 
  \node[white] (vn+p'')     at (4,-5) {$v_{n+p}''$};

  \draw[to] (v1) to node[auto] {$g_1$} (v2); 
  \draw[to] (v2) to node[auto] {$g_2$} (dots); 
  \draw[to] (dots) to node[auto] {$g_n$} (vn+1);
  \draw[to] (vn+1) to node[auto] {$g_{n+1}$} (dots2);
  \draw[to] (dots2) to node[auto] {$g_{n+p-1}$} (vn+p);  

  \draw[to, loop left] (v1) to node[auto] {$0 ,  \ldots , g_{1}-1$} (v1); 
  \draw[to, bend left = 25] (v2) to node[auto] {\quad \qquad \qquad $0, \ldots , g_{2}-1$} (v1);    
  \draw[to, bend left = 45] (vn+1) to node[auto] {\qquad $0, \ldots , g_{n+1}-1$} (v1);     
  \draw[to, bend left = 60] (vn+p) to node[auto] {$0, \ldots , g_{n+p}-1$} (v1);     
  \draw[to, bend right = 45] (vn+p) to node[auto,swap] {$g_{n+p}$} (vn+1);     

  \draw[to] (vn+1') to node[auto] {$g_{n+1}$} (dots2');
  \draw[to] (dots2') to node[auto] {$g_{n+p-1}$} (vn+p');  
  \draw[to, bend right = 45] (vn+p') to node[auto,swap] {$g_{n+p}$} (vn+1');     
  \draw[to] (vn+1') to node[very near start, left] {$\; 0, \ldots, g_{n+1}-1$} (v1);
  \draw[to] (vn+p') to node[near start, right] {\qquad $0, \ldots, g_{n+p}-1$} (v1);

  \draw[to] (vn+1'') to node[auto] {$g_{n+1}$} (dots2'');
  \draw[to] (dots2'') to node[auto] {$g_{n+p-1}$} (vn+p'');  
  \draw[to, bend left = 45] (vn+p'') to node[auto,swap] {$g_{n+p}$} (vn+1'');     
  \draw[to] (vn+1'') to node[very near start, left] {$\; 0, \ldots, g_{n+1}-1$} (v1);
  \draw[to, bend left = 10] (vn+p'') to node[very near start, right] {$0, \ldots, g_{n+p}-1$} (v1);
  
\end{tikzpicture}
\end{center}
\caption[Fiber product covers of sofic beta-shifts.]{Fiber product cover of a sofic beta-shift with minimal $n,p$ such that $g(\beta) = g_1 \cdots g_n(g_{n+1} \cdots g_{n+p})^\infty$.} 
\label{fig_beta_fiber}
\end{figure}
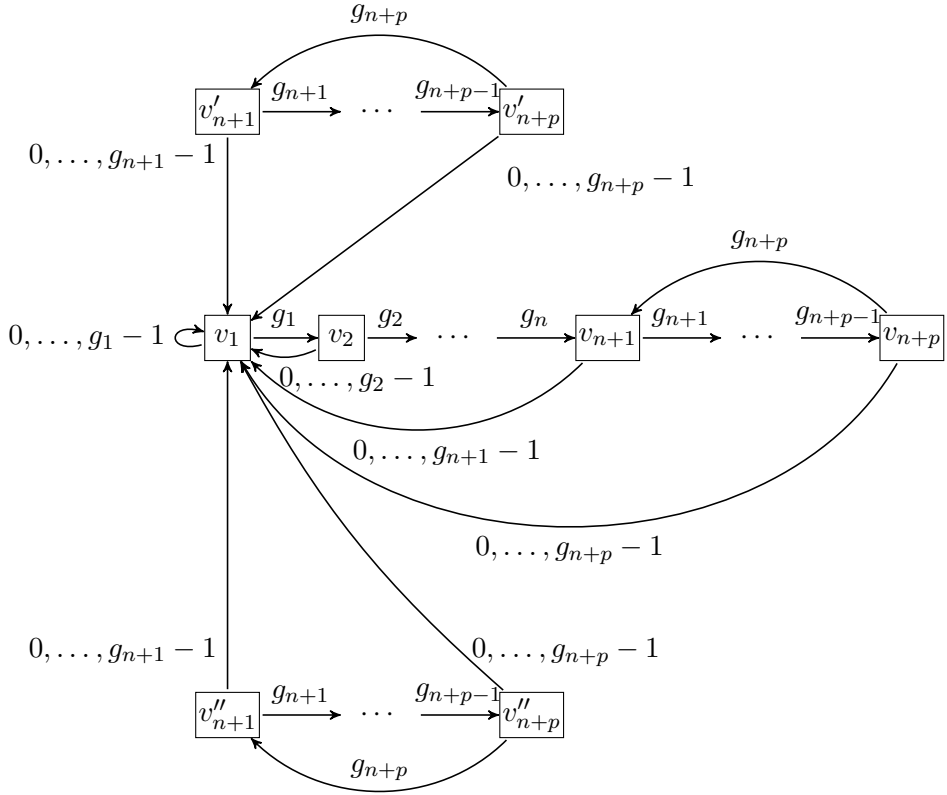

The following proposition shows that the fiber product cover always has the structure seen in the previous example.

\begin{proposition}
\label{prop_beta_fiber}
\index{beta-shift!fiber product cover of}
Let $\beta > 1$ such that $g(\beta)$ is eventually periodic but not periodic, then the fiber product cover of $\X_\beta$ is the graph shown in Figure \ref{fig_beta_fiber}.
\end{proposition}

\begin{proof}
Let $(F_\beta, \LL_\beta)$ be the right Fischer cover of $\X_\beta$, and let $n,p$ be minimal such that $g(\beta) = g_1 \cdots g_n (g_{n+1} \cdots g_{n+p})^\infty$. By the proof of Proposition \ref{prop_beta_2_to_1}, a left-ray will have a unique presentation unless it is equal to $w^\infty$ for some cyclic permutation $w$ of the period $g_{n+1} \cdots g_{n+p}$. To find the fiber product cover, it is therefore sufficient to consider such periodic left-rays. 

By the proof of Proposition \ref{prop_beta_2_to_1}, there exist $u_0, \ldots , u_{p-1}, u'_0, \ldots, u'_{p-1} \in F_\beta^0$ with $u_i \neq u_i'$ such that there are edges labeled $g_{n+i+1}$ from $u_i$ to $u_{i+1 \pmod{p}}$ and from $u'_i$ to $u'_{i+1 \pmod{p}}$ for each $0 \leq i \leq p-1$. Now $(u_0,u'_0), \ldots, (u_{p-1},u'_{p-1})$ are the only off-diagonal vertices in the fiber product cover. For each $0 \leq i \leq p-1$, the fiber product cover has an edge labeled $g_{n+i+1}$ from $(u_i,u'_i)$ to $(u_{i+1 \pmod{p}},u'_{i+1 \pmod{p}})$ and edges labeled $0, \ldots, g_{n+i+1}-1$ from $(u_i,u'_i)$ to $(v_1,v_1)$, where $v_1$ is the first vertex of the right Fischer cover. This gives the labeled graph shown in Figure \ref{fig_beta_fiber}.
\end{proof}

\begin{remark}
\label{rem_beta_fiber_action}
Let $\beta > 1$ with $g(\beta) = g_1 \cdots g_n(g_{n+1} \cdots g_{n+p})^\infty$ for minimal $n,p$, and let $(P_\beta, \LL_{P_\beta})$ be the right fiber product cover of $\X_\beta$. Consider the $\Z / 2\Z$ action on $\X_{P_\beta}$ induced by the labeling. The element $1 \in \Z / 2\Z$ acts by fixing the part of the graph that is isomorphic to the Fischer cover and switching the two analogous irreducible components lying above (see Figure \ref{fig_beta_fiber}).
\end{remark}

\section{Invariants and equivalences}
\label{sec_beta_classification}
\label{sec_beta_invariants}
In this section, the acquired knowledge about the structure of covers of sofic beta-shifts will be used to compute flow invariants and to reduce the flow classification problem through a series of concrete constructions. 

\subsection{Flow classification of beta-shifts of finite type}
The characterisation of beta-shifts of finite type given in Corollary \ref{cor_beta_sft} makes it possible to give a complete flow classification of such shifts.

\begin{proposition}
\label{prop_beta_sft_bf}
Given $\beta > 1$ such that $\X_\beta$ is an SFT, choose minimal $p \in \N$ such that the generating sequence is $g(\beta) = (g_1 \cdots g_p)^\infty$, then $\BF_+(\X_\beta) = - \Z / S\Z$ with $S=\sum_{j=1}^p g_j$. In particular, every SFT beta-shift is flow equivalent to a full shift.
\end{proposition}

\noindent
Note that $\sum_{j=1}^p g_j = \sum_{i=1}^p a_i - 1$ when $e(\beta) = a_1 \cdots a_p00\cdots$ and $a_p \neq 0$.

\begin{proof}
By Proposition \ref{prop_beta_sofic}, the (non-symbolic) adjacency matrix of the underlying graph of the right Fischer cover of $\X_\beta$ is
\begin{displaymath}
A = \begin{pmatrix}
g_1      & 1 & 0 & \cdots & 0 & 0 \\ 
g_2      & 0 & 1 &             & 0 & 0 \\ 
g_3      & 0 & 0 &             & 0 & 0 \\ 
\vdots   &    &    & \ddots &    &    \\ 
g_{p-1 }& 0 & 0 &          & 0 & 1 \\ 
g_p+1      & 0 & 0 & \cdots & 0 & 0  
\end{pmatrix}.
\end{displaymath} 
Now it is straightforward to compute the complete invariant by finding the Smith normal form and determinant of $\Id - A$.
\end{proof}

\noindent
It is also not hard to construct a concrete flow equivalence between the beta-shift considered in Proposition \ref{prop_beta_sft_bf} and the full $(S+1)$-shift.

\begin{example}
If $\beta > 1$, then the entropy of $\X_\beta$ is $\log \beta$ \cite{parry,renyi}.
In particular, beta-shifts $\X_{\beta_1}$ and $\X_{\beta_2}$ with $\beta_1 \neq \beta_2$ are never conjugate.
However, by Theorem \ref{thm_beta_existence}, there exist $1 < \beta _1 < 2$ and $2 < \beta_2 < 3$ such that $(110)^\infty$ is the generating sequence of $\X_{\beta_1}$ and $(20)^\infty$ is the generating sequence of $\X_{\beta_2}$, and the beta-shifts $\X_{\beta_1}$ and $\X_{\beta_2}$ are flow equivalent by Proposition \ref{prop_beta_sft_bf}. 
\end{example}

\subsection{Bowen-Franks groups}
The Bowen-Franks groups of the underlying graphs of the covers from Section \ref{sec_beta_covers} will be computed in this section. 

\begin{proposition} 
\label{prop_beta_bf}
Let $\beta > 1$ with sofic $\X_\beta$, and let $n,p$ be minimal such that $g(\beta) = g_1 \cdots g_n ( g_{n+1} \cdots g_{n+p})^\infty$. Let $A_F$
and $A_P$ be the adjacency matrices of the underlying graphs of respectively the Fischer cover
and the fiber product cover. Then $\BF_+(A_F) = -\Z / S\Z$
and $\BF(A_P) = (\Z / S\Z) \oplus \Z \oplus \Z$ where $S = \sum_{i=1}^p g_{n+i}$.
\end{proposition}

\begin{proof}
By Proposition \ref{prop_beta_sofic},
\begin{displaymath}
A_F = 
\left(\!\! \begin{array}{ c c c c c c | c c c c c c  }
g_1    & 1 & 0 & \cdots & 0 & 0 &
0        & 0 & 0 & \cdots & 0 & 0  \\
g_2    & 0 & 1 &            & 0 & 0 &
0        & 0 & 0 &            & 0 & 0   \\
g_3    & 0 & 0 &            & 0 & 0 &
0        & 0 & 0 &           & 0 & 0    \\
\vdots &   &    & \ddots &    & \vdots &
\vdots &   &    & \ddots &    & \vdots  \\
g_{n-1}    & 0 & 0 &           & 0 & 1 &
0             & 0 & 0 &            & 0 & 0  \\
g_n         & 0 & 0 & \cdots & 0 & 0 &
1             & 0 & 0 & \cdots & 0 & 0  \\
\hline
g_{n+1} &    &    &            &    &    &
0           & 1 &  0 & \cdots & 0 & 0  \\
g_{n+2} &    &    &            &   &    &
0           & 0 & 1 &            & 0 & 0   \\
g_{n+3} &    &    &            &    &   &
0           & 0 & 0 &           & 0 & 0   \\
\vdots   &   &    & 0         &    &    &
\vdots   &   &    & \ddots &    & \vdots  \\
g_{n+p-1}    &  &  &           &    &   &
0             & 0 & 0 &            & 0 & 1   \\
g_{n+p}  &    &    &            &    &    &
1             & 0 & 0 & \cdots & 0 & 0   
\end{array}\!\! \right).
\end{displaymath}
It is straightforward to find the invariant by computing the Smith normal form and determinant of $\Id - A_F$. The other Bowen-Franks group is computed in the same manner.
\end{proof}

\noindent
Note that these Bowen-Franks groups only contain information about the sum of the numbers in the periodic part of the generating sequence. This is partially explained by the results of the following section. 

\subsection{Concrete constructions}
\label{sec_beta_constructions}
This section contains recipes for concrete constructions of flow equivalences reducing the complexity of beta-shifts. Let $n \in \N$ and $n-1 < \beta < n$ be given, and let $X = \X_\beta$. Define $\varphi \colon \AA_\beta \to \{ 0 ,1\}^*$ by 
\begin{displaymath}
   \varphi(j) = 
   1^j 0 
\end{displaymath}
and extend this to $\varphi \colon \AA_\beta^* \to \{ 0 ,1\}^*$ by $\varphi(a_1 \cdots a_k) = \varphi(a_1) \cdots \varphi(a_k)$.
It is straightforward to check that the shift closure of $\varphi(\BB(X))$ is the language of a shift space $X'$ flow equivalent to $X$ (see \cite[Section 2.3]{johansen_thesis} for a collection of similar constructions). Define the induced map $\varphi \colon X^+ \to X'^+$ by $\varphi(x_1 x_2 \cdots ) = \varphi(x_1) \varphi(x_2) \cdots $.

\begin{lemma}
\label{lem_beta_order_preserving}
The map $\varphi \colon X^+ \to X'^+$ constructed above is surjective and order preserving.
\end{lemma}

\begin{proof}
The map preserves the lexicographic order by construction.
Let $x'^+ \in X'^+$ be given. By construction, there exists $x^+ = x_1x_2 \cdots \in X^+$ such that $x'^+ = w \varphi(x_2) \varphi(x_3) \cdots$ where $w$ is a suffix of $\varphi(x_1)$. Now there exists $1 \leq a \leq x_1$ such that $w = \varphi(a)$. Since $a x_2 x_3 \cdots \leq  x_1 x_2 \cdots \leq g(\beta)$, it follows from Theorem \ref{thm_renyi} that $a x_2 x_3 \cdots \in X^+$, so $x'^+ \in \varphi(X^+)$.  
\end{proof}

\begin{theorem}\index{beta-shift!flow equivalence of}
For every $\beta > 1$ there exists $1 < \beta' < 2$ such that $\X_\beta \FE \X_{\beta '}$. 
\end{theorem}

\begin{proof}
Let $X = \X_\beta$, and construct $X' \FE X$ and $\varphi \colon X^+ \to X'^+$ as above. Use Theorem \ref{thm_beta_existence} to choose $1 < \beta' < 2$ such that $g(\beta') = \varphi(g(\beta))$. The aim is to prove that $X' = \X_{\beta'}$.
  
Given $x'^+ = x'_1 x'_2 \cdots \in X'^+$ and $k \in \N$, let $x'_{k^+} = x'_k x'_{k+1} \cdots$. Use Lemma \ref{lem_beta_order_preserving} to choose $x^+ \in X^+$ such that $\varphi(x^+) = x'_{k^+}$. Now $x^+ \leq g(\beta)$ and $\varphi$ is order preserving, so $x'_{k^+} \leq g(\beta')$. By Theorem \ref{thm_renyi}, this means that $x'^+ \in \X_{\beta'}^+$. 

Let $x'^+ = x'_1 x'_2 \cdots \in \X_{\beta'}^+$ and let $n = \max \AA_\beta$. Consider the extension $\varphi \colon \{0, \ldots, n\}^\N \to \{0, 1\}^\N$ and note that $x'^+$ does not contain $1^{n+1}$ as a factor, so there exists $x^+ = x_1 x_2 \cdots \in \{0, \ldots, n\}^\N$ such that $\varphi(x^+) = x'^+$.
Let $k \in \N$ be given and let $x_{k^+} = x_k x_{k+1} \cdots$. Then there exists $l \geq k$ such that 
$\varphi(x_{k^+}) =  x'_l x'_{l+1} \cdots \leq g(\beta') = \varphi(g(\beta))$, but $\varphi$ is order preserving, so this means that $x_{k^+} \leq g(\beta)$. Hence, $x^+ \in \X_\beta^+$ and $x'^+ \in X'^+$ as desired.
\end{proof}

\noindent
This shows that it is sufficient to consider $1 < \beta < 2$ when trying to classify sofic beta-shifts up to flow equivalence. The next goal is to find a standard form that any sofic beta-shift can be reduced to.

\begin{lemma}
\label{lem_beta_delete_0_in_beginning}
Let $1 < \beta < 2$ such that $g(\beta)$ is aperiodic and let $n$ be the largest number such that $1^n$ is a prefix of $g(\beta)$. Then $\X_\beta \FE 
\X_{\beta'}$ where $g(\beta')$ is obtained from $g(\beta)$ by deleting a $0$ immediately after each occurrence of $1^n$.
\end{lemma} 

\begin{proof}
Note that $1^{n+1} \notin \BB(\X_\beta)$, so each occurrence of $1^n$ in $\X_\beta$ is followed by $0$.
Define a map $\phi_{1^n0 \mapsto 1^n} \colon \BB(\X_\beta) \to \{0,1\}^\ast$ by mapping each word $w$ to the word obtained by deleting the 0 in each occurrence of $1^n0$. It is straightforward to prove that $\phi_{1^n0 \mapsto 1^n}(\BB(\X_\beta))$ is the language of a shift space $\X_\beta^{1^n0 \mapsto 1^n}$  (see \cite[Section 2.3]{johansen_thesis}). There is an upper bound on $\{ k \mid (1^n 0)^k \in \BB(\X_\beta) \}$ since $1^n0$ is a prefix of $g(\beta)$ which is aperiodic, and therefore $\X_\beta \FE \X_\beta^{1^n0 \mapsto 1^n}$ by arguments analogous to the ones used in \cite[Section 2.4]{johansen_thesis}. 

Define $\varphi \colon \X_\beta^+ \to \{ 0, 1 \}^\N$ such that $\varphi(x^+)$ is the sequence obtained from $x^+$ by deleting one $0$ immediately after each occurrence of $1^n$. This is an order preserving map. Use Theorem \ref{thm_beta_existence} to choose $\beta'$ such that $g(\beta') = \varphi(g(\beta))$.
Given $y^+ \in \X_\beta^{1^n0 \mapsto 1^n}$ and $k \in \N$, define $x^+$ to be the sequence obtained from $y_k y_{k+1} \cdots$ by inserting $0$ after $y_j$ if there exists $l \in \N$ such that $y_{[j-ln,j]} = 01^{ln}$ or such that $j = k+ln-1$ and $y_{[k,j]} = 1^{ln}$.
Now, $x^+ \in \X_\beta^+$, so $y_k y_{k+1} \cdots = \varphi(x^+) \leq \varphi(g(\beta)) = g(\beta')$, so $y^+ \in \X_{\beta'}^+$ by Theorem \ref{thm_renyi}. A similar argument proves the other inclusion.  
\end{proof}

\begin{corollary}
\label{cor_beta_add_0_in_beginning}
\index{beta-shift!flow equivalence of}
Let $1 < \beta < 2$ such that $g(\beta)$ is aperiodic and let $n$ be the largest number such that $1^n$ is a prefix of $g(\beta)$. For each $k > n / 2$, $\X_\beta \FE 
 \X_{\beta'}$ where $g(\beta')$ is obtained from $g(\beta)$ by inserting a $0$ immediately after the initial $1^k$ and after each subsequent occurrence of $01^k$.
\end{corollary}

\begin{proof}
Apply Lemma \ref{lem_beta_delete_0_in_beginning} to $\X_{\beta'}$. 
\end{proof}

\begin{example}
\label{ex_beta_standard}
Consider $\beta > 1$ such that 
$g(\beta) = 1101101(0101100)^\infty$. Use Lemma \ref{lem_beta_delete_0_in_beginning} to see that $\X_\beta \FE \X_{\beta_1}$ when 
$g(\beta_1) = 11111(010110)^\infty$.
By Corollary \ref{cor_beta_add_0_in_beginning},  
$\X_{\beta_1} \FE \X_{\beta_2}$ when
\begin{displaymath} 
g(\beta_2) = 111110(010110)^\infty = 111(110010)^\infty.
\end{displaymath}
Note how this operation permutes the period of the generating sequence. Use Corollary \ref{cor_beta_add_0_in_beginning} again to show that $\X_{\beta_2} \FE \X_{\beta_3}$ when 
\begin{displaymath}
g(\beta_3) = 1110(110010)^\infty = 11(101100)^\infty.
\end{displaymath}
An additional application of Corollary \ref{cor_beta_add_0_in_beginning} will not reduce the aperiodic beginning of the generating sequence further, since it will also add an extra $0$ inside the period.
Note in particular that the length of neither the period nor the beginning of the generating sequence is a flow invariant. The sum of entries in the period of the generating sequence is a flow invariant by Proposition \ref{prop_beta_bf}, but the same is apparently not true for the sum of entries in the aperiodic beginning. Indeed, it is straightforward to use Lemma \ref{lem_beta_delete_0_in_beginning} and Corollary \ref{cor_beta_add_0_in_beginning} to show that $\X_\beta \FE \X_{\beta'}$ if
\begin{align*}
   g(\beta') &= 1^{3n}(110010)^\infty \textrm{ or } \\ 
   g(\beta') &= 1^{3n+2}(101100)^\infty   
\end{align*}
for some $n \in \N$. However, at this stage it is, for instance, still unclear whether $\X_\beta \FE \X_{\beta''}$ when $g(\beta'') = 1^{3n}(101100)^\infty$.
\end{example}


The following proposition shows how the tools applied in Example \ref{ex_beta_standard} can be used to modify the generating sequence of a general sofic beta-shift to produce a beta-shift flow equivalent to the original.

\begin{proposition}\label{prop_beg_and_per}
Let $1 < \beta < 2$ with $g(\beta) = \beg \per^\infty$ where $\beg =  g_1 \cdots g_n$ and $\per = g_{n+1} \cdots g_{n+p} = 1^{p_1} 0^{q_1} \cdots 1^{p_m} 0^{q_m}$. Assume that $n$ and $p$ are minimal, and let $S_\beg  = \sum_{i=1}^n g_i$, and $ S_\per = \sum_{i=1}^p g_{n+i} = \sum_{j=1}^m p_j$. Given $1 \leq k \leq m$, $\X_\beta \FE \X_{\beta'}$ when 
\begin{displaymath}
g(\beta') = 1^{S_\beg+lS_\per+p_1 + \cdots + p_k}(1^{p_{k+1}} 0^{q_{k+1}} \cdots 1^{p_m} 0^{q_m} 1^{p_1} 0^{q_1} \cdots 1^{p_k} 0^{q_k} )^\infty
\end{displaymath}
for $l \in \Z$ with $S_\beg+lS_\per+p_1 + \cdots + p_k > 0$.
\end{proposition}

\begin{proof}
By Lemma \ref{lem_beta_delete_0_in_beginning}, $\X_\beta$ is flow equivalent to the beta-shift with generating sequence $1^{S_\beg}\per$. The rest of the statement follows by applying Lemma \ref{lem_beta_delete_0_in_beginning} and Corollary \ref{cor_beta_add_0_in_beginning} as in Example \ref{ex_beta_standard}.
\end{proof}

\begin{remark}
\label{rem_beta_standard}
In particular, Proposition \ref{prop_beg_and_per} can be used to show that a beta-shift $\X_\beta$ with $g(\beta) = \beg \per^\infty$ is flow equivalent to some beta-shift $\X_{\beta'}$ where $g(\beta') = 1^n (\per')^\infty$ with $n \leq \lvert \per' \rvert$.
\end{remark}

\begin{remark}\label{rem_reduction_of_problem}
By proposition \ref{prop_beta_bf}, two flow equivalent sofic beta-shifts must have the same sum of elements in the period, but it is currently unknown whether this number is a complete invariant.

Assume that $\X_\beta \FE \X_{\beta'}$ whenever $g(\beta) = \beg\per^\infty$ and $g(\beta') = \beg(\per')^\infty$ with $\per'$ obtained from $\per$ by inserting a 0 after an arbitrary 1 in $\per$ and $\per'$ minimal in the sense that no shorter period gives the same generating sequence. 
Then it follows from Proposition \ref{prop_beg_and_per} that two arbitrary beta-shifts are flow equivalent if and only if they have the same sum of entries in the period. 
In this way, the flow classification of sofic beta-shifts can be reduced to the question of whether it is possible to add zeroes in the period as specified above using conjugacies, symbol expansions and symbol reductions. However, it is still unknown whether it is possible to carry out this final reduction of the problem to show that the sum of entries in the period of the generating sequence is a complete invariant of flow equivalence of sofic beta-shifts.
\end{remark}

\subsection{Flow equivalence of fiber products}
Let $1 < \beta < 2$ with sofic $\X_\beta$, let $n,p$ be minimal such that $g(\beta) = g_1 \cdots g_n (g_{n+1} \cdots  g_{n+p})^\infty$, and let $(P_\beta, \LL_\beta)$ be the fiber product cover of $\X_\beta$. The goal of this section is to study the flow class of the reducible edge shift $\X_{P_\beta}$ defined by the underlying graph. Let $N = \sum_{i=1}^n g_i$ and $S = \sum_{i=1}^p g_{n+i}$. By Remark \ref{rem_beta_standard}, it can be assumed that $N \leq S$ without loss of generality.

Let $P_\beta^0 = \{v_1, \ldots, v_{n+p}, v'_{n+1}, \ldots, v'_{n+p}, v''_{n+1}, \ldots, v'_{n+p} \}$ as in Figure \ref{fig_beta_fiber}. Let $e \in P_\beta^1$ with $\LL_\beta(e) = 0$ and $r(e) \neq v_1$. Then there exists $f \in P_\beta^1$ such that for $\lambda \in \X_{P_\beta}$, $\lambda_0 = f$ if and only if $\lambda_1=e$. Hence, all these edges can be removed using symbol contraction, and this leaves the edge shift of the graph shown in Figure \ref{fig_beta_fiber_fe_I}.
In this graph, the vertices $v_N$ and $v_{N+S}$ each emit one edge to $v_{N+1}$ and one edge to $v_1$. Use in-amalgamation to merge these two vertices. This identifies the edges $e$ and $g$ and the edges $f$ and $h$ marked in Figure \ref{fig_beta_fiber_fe_I}. The result is a graph of the same form, where the size of $N$ is reduced by $1$. Repeat this process $N$ times to show that $\X_{P_\beta}$ is flow equivalent to the graph in Figure \ref{fig_beta_fiber_fe_II}. This leads to the following:

\begin{figure}
\begin{center}
\begin{tikzpicture}
 [bend angle=10,
   clearRound/.style = {circle, inner sep = 0pt, minimum size = 17mm},
   clear/.style = {rectangle, minimum width = 10 mm, minimum height = 6 mm, inner sep = 0pt},  
   greyRound/.style = {circle, draw, minimum size = 1 mm, inner sep = 0pt, fill=black!10},
 grey/.style = {rectangle, draw, minimum size = 6 mm, inner sep = 1pt, fill=black!10},
   white/.style = {rectangle, draw, minimum size = 6 mm, inner sep =1pt},
   to/.style = {->, shorten <= 1 pt, >=stealth', semithick}]
  
  \node[white] (v1)     at (0,0) {$v_1$};
  \node[white] (v2)     at (1.8,0) {$v_2$};
  \node[clear] (dots)   at (3.6,0) {$\cdots$};
  \node[white] (vn) at (5.4,0) {$v_{N}$};  
  \node[white] (vn+1) at (7.2,0) {$v_{N+1}$};
  \node[clear] (dots2) at (9,0) {$\cdots$}; 
  \node[white] (vn+p) at (10.8,0) {$v_{N+S}$};

  \node[white] (vn+1')     at (0,2) {$v_{N+1}'$};
  \node[clear] (dots2')     at (1.5,2) {$\cdots$}; 
  \node[white] (vn+p')     at (3,2) {$v_{N+S}'$};

  \node[white] (vn+1'')     at (0,-3) {$v_{N+1}''$};
  \node[clear] (dots2'')     at (1.5,-3) {$\cdots$}; 
  \node[white] (vn+p'')     at (3,-3) {$v_{N+S}''$};

  \draw[to] (v1) to node[auto] {$$} (v2); 
  \draw[to] (v2) to node[auto] {$$} (dots); 
  \draw[to] (dots) to node[auto] {$$} (vn);
  \draw[to] (vn) to node[auto] {$e$} (vn+1);  
  \draw[to] (vn+1) to node[auto] {$$} (dots2);
  \draw[to] (dots2) to node[auto] {$$} (vn+p);  

  \draw[to, loop left] (v1) to node[auto] {$$} (v1); 
  \draw[to, bend left = 25] (v2) to node[auto] {$$} (v1);    
  \draw[to, bend left = 30] (vn) to node[near start, below] {$f$} (v1);
  \draw[to, bend left = 35] (vn+1) to node[near start, below] {$$} (v1);       
  \draw[to, bend left = 35] (vn+p) to node[auto] {$h$} (v1);     
  \draw[to, bend right = 45] (vn+p) to node[auto,swap] {$g$} (vn+1);     

  \draw[to] (vn+1') to node[auto] {$$} (dots2');
  \draw[to] (dots2') to node[auto] {$$} (vn+p');  
  \draw[to, bend right = 45] (vn+p') to node[auto,swap] {$$} (vn+1');     
  \draw[to] (vn+1') to node[very near start, left] {$$} (v1);
  \draw[to] (vn+p') to node[near start, right] {$$} (v1);

  \draw[to] (vn+1'') to node[auto] {$$} (dots2'');
  \draw[to] (dots2'') to node[auto] {$$} (vn+p'');  
  \draw[to, bend left = 45] (vn+p'') to node[auto,swap] {$$} (vn+1'');     
  \draw[to] (vn+1'') to node[very near start, left] {$$} (v1);
  \draw[to] (vn+p'') to node[very near start, right] {$$} (v1);
\end{tikzpicture}
\end{center}
\caption[Flow equivalence of fiber product covers I.]{An unlabeled graph defining an edge shift flow equivalent to the edge shift of the underlying graph of the fiber product cover of a sofic beta-shift with minimal $n,p$ such that  $g(\beta) =  g_1 \cdots g_n (g_{n+1} \cdots g_{n+p})^\infty$. Here $N = \sum_{i=1}^n g_i$ and $S = \sum_{i=1}^p g_{n+i}$. Every vertex emits precisely two edges, one of which terminates at $v_1$.} 
\label{fig_beta_fiber_fe_I}
\end{figure}
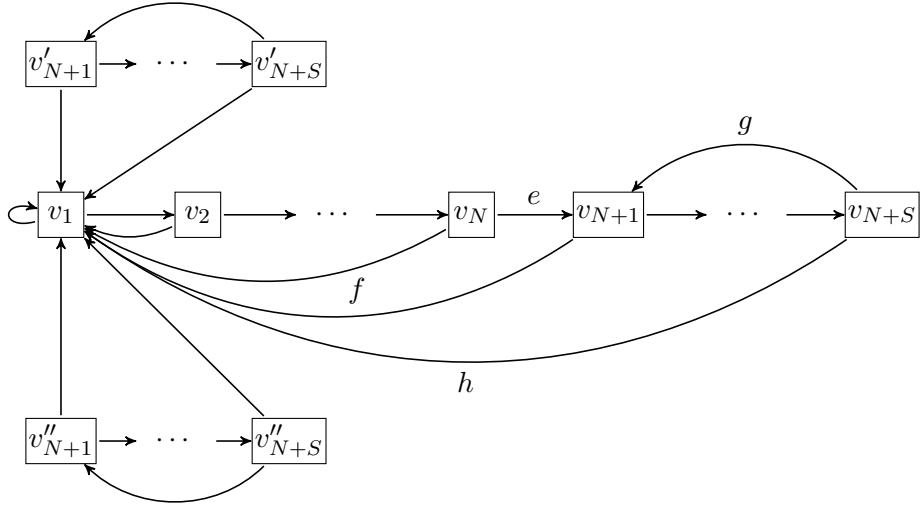

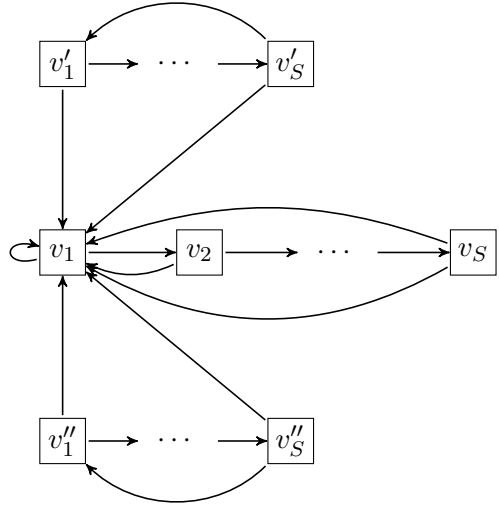
\begin{figure}
\begin{center}
\begin{tikzpicture}
 [bend angle=10,
   clearRound/.style = {circle, inner sep = 0pt, minimum size = 17mm},
   clear/.style = {rectangle, minimum width = 10 mm, minimum height = 6 mm, inner sep = 0pt},  
   greyRound/.style = {circle, draw, minimum size = 1 mm, inner sep =
      0pt, fill=black!10},
   grey/.style = {rectangle, draw, minimum size = 6 mm, inner sep =
      1pt, fill=black!10},
   white/.style = {rectangle, draw, minimum size = 6 mm, inner sep =
      1pt},
   to/.style = {->, shorten <= 1 pt, >=stealth', semithick}]
  
  \node[white] (v1)     at (0,0) {$v_1$};
  \node[white] (v2)     at (1.8,0) {$v_2$};
  \node[clear] (dots)   at (3.6,0) {$\cdots$};
  \node[white] (vp) at (5.4,0) {$v_{S}$};

  \node[white] (vn+1')     at (0,2.5) {$v_{1}'$};
  \node[clear] (dots2')     at (1.5,2.5) {$\cdots$}; 
  \node[white] (vn+p')     at (3,2.5) {$v_{S}'$};

  \node[white] (vn+1'')     at (0,-2.5) {$v_{1}''$};
  \node[clear] (dots2'')     at (1.5,-2.5) {$\cdots$}; 
  \node[white] (vn+p'')     at (3,-2.5) {$v_{S}''$};

  \draw[to] (v1) to node[auto] {$$} (v2); 
  \draw[to] (v2) to node[auto] {$$} (dots); 
  \draw[to] (dots) to node[auto] {$$} (vp);

  \draw[to, loop left] (v1) to node[auto] {$$} (v1); 
  \draw[to, bend left = 25] (v2) to node[auto] {$$} (v1);    
  \draw[to, bend left = 30] (vp) to node[near start, below] {$$} (v1);
  \draw[to, bend right = 20] (vp) to node[near start, below] {$$} (v1);

  \draw[to] (vn+1') to node[auto] {$$} (dots2');
  \draw[to] (dots2') to node[auto] {$$} (vn+p');  
  \draw[to, bend right = 45] (vn+p') to node[auto,swap] {$$} (vn+1');     
  \draw[to] (vn+1') to node[very near start, left] {$$} (v1);
  \draw[to] (vn+p') to node[near start, right] {$$} (v1);

  \draw[to] (vn+1'') to node[auto] {$$} (dots2'');
  \draw[to] (dots2'') to node[auto] {$$} (vn+p'');  
  \draw[to, bend left = 45] (vn+p'') to node[auto,swap] {$$} (vn+1'');     
  \draw[to] (vn+1'') to node[very near start, left] {$$} (v1);
  \draw[to] (vn+p'') to node[very near start, right] {$$} (v1);
\end{tikzpicture}
\end{center}
\caption[Flow equivalence of fiber product covers II.]{An unlabeled graph defining an edge shift flow equivalent to the edge shift of the underlying graph of the fiber product cover of a sofic beta-shift with minimal $n,p$ such that  $g(\beta) =  g_1 \cdots g_n (g_{n+1} \cdots g_{n+p})^\infty$. Here $S = \sum_{i=1}^p g_{n+i}$.} 
\label{fig_beta_fiber_fe_II}
\end{figure}

\begin{proposition}
\label{prop_beta_fiber_fe}
For $i \in \{1,2\}$, let $\beta_i > 1$ with minimal $n_i,p_i$ such that
$g(\beta) =  g^i_1 \cdots g^i_{n_i} (g^i_{n_i+1} \cdots g^i_{n_i+p_i})^\infty$,
let $S_i = \sum_{i=1}^{p_i} g_{n_i+i}$, and let $(P_{\beta_i}, \LL_{P_{\beta_i}})$ be the fiber product cover of $\X_{\beta_i}$. Then there exists a flow equivalence $\Phi \colon S\X_{P_{\beta_1}} \to S\X_{P_{\beta_2}}$ which commutes with the $\Z /2\Z$ actions induced by the labels if and only if $S_1 = S_2$.
\end{proposition}

\begin{proof}
If there is such a flow equivalence, then the Bowen-Franks groups of the edge shifts $\X_{P_{\beta_1}}$ and $\X_{P_{\beta_2}}$ must be equal \cite{franks}, so $S_1 = S_2$ by Proposition \ref{prop_beta_bf}. Reversely, if $S_1 = S_2$, then the preceding arguments prove that there is a flow equivalence $\Phi \colon S\X_{P_{\beta_1}} \to S\X_{P_{\beta_2}}$. Furthermore, the $\Z / 2\Z$ actions on $\X_{P_{\beta_i}}$ induced by the labels are respected by the conjugacies and symbol reductions used in the construction, so $\Phi$ commutes with the actions. 
\end{proof}

\section{Perspectives}
The concrete constructions of flow equivalences culminating in Remark \ref{rem_reduction_of_problem} give a drastic reduction of the flow classification problem for sofic beta-shifts by showing that the general problem can be reduced to a question of whether it is possible to add extra zeroes inside the period of the generating sequence. However, it is still unknown whether this final step can be carried out in general.
For example, there is no known way to determine whether the beta-shifts $\X_{\beta_1}$ and $\X_{\beta_2}$ are flow equivalent when $\beta_1, \beta_2$ have generating sequences  
\begin{displaymath}
  g(\beta_1) = 1(110)^\infty
\textrm{ and }
  g(\beta_2) = 11(110)^\infty . 
\end{displaymath}
This is arguably the simplest question about the flow equivalence of sofic beta-shifts that cannot be answered directly by the tools presented above.

Sofic beta-shifts do not belong to the class of irreducible 2-sofic shifts currently covered by the classification results obtained in \cite{boyle_carlsen_eilers_isotopy,boyle_carlsen_eilers_sofic}. However, if it is possible to extend these results to cover sofic beta-shifts then Proposition \ref{prop_beta_fiber_fe} will be sufficient to prove that  the sum of elements in the period of the generating sequence is a complete invariant of flow equivalence of sofic beta-shifts.

\end{document}